\documentclass[12pt]{article}

\usepackage{amssymb}
\usepackage{latexsym}
\usepackage{amsmath}
\usepackage{indentfirst}
\usepackage{float}
\usepackage{hyperref}
\usepackage{color}

\title{All Ramsey critical graphs for a large tree versus $tK_{m}$}
\author{Zhiyu Cheng \thanks{School of Mathematics and Statistics, Hainan University, Haikou, Hainan 570028, PR China. Email: 3426315139@qq.com}
\and Zhidan Luo \thanks{School of Mathematics and Statistics, Hainan University, Haikou, Hainan 570028, PR China. Research supported by the National Natural Science Foundation of China (No.12401449), the Hainan Provincial Natural Science Foundation of China (No.125QN209) and the Hainan University Research Foundation Project (No. KYQD(ZR)-23155). Email: luodan@hainanu.edu.cn}
\and Pingge Chen \thanks{Corresponding author. College of Science, Hunan University of Technology, Zhuzhou, Hunan 412007, PR China. Research supported by the National Natural Science Foundation of China (No.12101221). Email:chenpingge@hut.edu.cn}
}
\date{}

\newtheorem{theo}{Theorem}[section]

\newtheorem{remark}[theo]{Remark}
\newtheorem{lemma}[theo]{Lemma}
\newtheorem{claim}[theo]{Claim}

\def\q{\hspace*{\fill}$\Box$\medskip}

\begin{document}
\maketitle
\begin{abstract}
Let $H, H_{1}$ and $H_{2}$ be graphs, and let $H\rightarrow (H_{1}, H_{2})$ denote that any red-blue coloring of $E(H)$ yields a red copy of $H_{1}$ or a blue copy of $H_{2}$. The Ramsey number for $H_{1}$ versus $H_{2}$, $r(H_{1}, H_{2})$, is the minimum integer $N$ such that $K_{N}\rightarrow (H_{1}, H_{2})$. The Ramsey critical graph $H$ for $H_{1}$ versus $H_{2}$ is a red-blue edge-colored $K_{N- 1}$ such that $H\not\rightarrow (H_{1}, H_{2})$, where $N= r(H_{1}, H_{2})$. In this paper, we characterize all Ramsey critical graphs for a large tree versus $tK_{m}$. As a corollary, we determine the star-critical Ramsey number for a large tree versus $tK_{m}$.


\end{abstract}

\section{Introduction}
  All graphs in this paper are simple graphs. For a graph $G$, let $V(G)$ and $E(G)$ be the vertex set and the edge set of $G$, respectively. Let $K_{n}$ be the complete graph of $n$ vertices. For graphs $H, H_{1}$ and $H_{2}$, let $H\rightarrow (H_{1}, H_{2})$ denote that any red-blue coloring of $E(G)$ yields a red copy of $H_{1}$ or a blue copy of $H_{2}$. The {\it Ramsey number} for $H_{1}$ versus $H_{2}$, $r(H_{1}, H_{2})$, is the minimum integer $N$ such that $K_{N}\rightarrow (H_{1}, H_{2})$. Let $v(G)$ and $\chi(G)$ be the number of vertices of $G$ and the chromatic number of $G$, respectively. Let $s(G)$ be the chromatic surplus of $G$, which is the cardinality of a minimum color class overall the proper colorings of $G$ with $\chi(G)$ colors. In 1981, Burr proved the following lower bound about the Ramsey number.
  \begin{theo}[{\rm Burr \cite{B}}]\label{theo1.1}
    For a connected graph $G$ and a graph $H$ with $v(G)\geq s(H)$,
    $$r(G,H) \geq (v(G)-1)(\chi(H)-1)+s(H).$$
  \end{theo}
  If the equality holds, then we say that $G$ is $H$-good. Let $T_{n}$ be a tree of $n$ vertices. In 1977, Chv\'atal proved that
  \begin{theo}[{\rm Chv\'atal \cite{C}}]\label{theo1.2}
    $T_{n}$ is $K_m$-good for all positive integers $n$ and $m$.
  \end{theo}
  In 1983, Burr and Erd\H{o}s \cite{BE} proved that for positive integers $k$ and $m$, there is an integer $n_{0}$ such that the family of connected graphs with bandwidth at most $k$ and at least $n_{0}$ vertices is $K_{m}$-good, where the bandwidth of a graph $G$ is the smallest integer $\ell$ such that there is an ordering $v_{1}, \dots, v_{n}$ of $V(G)$ satisfying $|i- j|\leq \ell$ for each edge $v_{i}v_{j}\in E(G)$. In 2013, Allen, Brightwell and Skokan \cite{ABS} extended it. They proved that for a fixed integer $\Delta$ and a graph $H$, there are integers $c$ and $n_{0}$ such that if $G$ is a connected graph of at least $n_{0}$ vertices with bandwidth at most $cn$ and maximum degree $\Delta$, then $G$ is $H$-good. In 2017, Pokrovskiy and Sudakov \cite{PS} proved that if $H$ is a graph and $n\geq 4v(H)$, then the path of $n$ vertices is $H$-good. Furthermore, Balla, Pokrovskiy and Sudakov \cite{BPS} proved that for all fixed integers $\Delta$ and $k$, there is a constant $C_{\Delta,k}$ such that if the maximum degree of $T_{n}$ is at most $\Delta$ and $H$ is a graph with $\chi(H)= k$ satisfying $n\geq C_{\Delta, k}v(H)\log^{4}v(H)$, then $T_{n}$ is $H$-good. For graphs $G$ and $H$, let $G\cup H$ be the union of vertex-disjoint copies of $G$ and $H$. For a positive integer $t$, let $tG$ be the union of $t$ vertex-disjoint copies of $G$. Let $\alpha(G)$ be the maximum size of independent sets of $G$. In 1975, Burr, Erd\H{o}s and Spencer\cite{BES} proved that $mv(G)+nv(H)- \min\{m\alpha (G),n\alpha (H)\}-1\leq r(mG,nH)\leq mv(G)+nv(H)-\min\{m\alpha (G),n\alpha (H)\}+c_{1}$, where $c_{1}$ is a constant depending only on $G$ and $H$. Furthermore, Burr \cite{B1} proved that $r(G,nH)=nv(H)+r(\mathcal{D}(G),H)-1$ for sufficiently large $n$, where $H$ is a connected graph and $\mathcal{D}(G)$ is a set of all graphs formed from $G$ by removing a maximal independent set. In 2023, Theorem \ref{theo1.2} was extended to $tK_{m}$.
  \begin{theo}[\rm{Hu and Peng \cite{HP1}, Luo and Peng \cite{LP}}]\label{theo1.3}
    Let $n, m$ and $t$ be positive integers. If $n$ is sufficiently large, then $T_{n}$ is $tK_{m}$-good.
  \end{theo}
  Recently, Hu and the second author \cite{HL} extended $tK_{m}$ to the union of vertex-disjoint copies of different sizes of complete graphs. Zhang and Chen \cite{ZC} extended $T_{n}$ to a sparsely connected graph. For more results on the Ramsey number, we recommend \cite{CFS} and \cite{R}.

  Let $H$ be a subgraph of $G$ and let $G- H$ be the graph obtained from $G$ by removing a copy of $H$, that is, $V(G- H)= V(G)$ and $E(G- H)= E(G)\backslash E(H)$. Let $K_{1, n}$ be the star of $n+ 1$ vertices. In 2010, Hook \cite{H1} introduced the {\it star-critical Ramsey number} for $H_{1}$ versus $H_{2}$, $r_{*}(H_{1}, H_{2})$, which is the minimum integer $k$ such that $K_{N}- K_{1,N- 1- k}\rightarrow (H_{1}, H_{2})$, where $N= r(H_{1}, H_{2})$. For more results on the star-critical Ramsey number, we recommend \cite{Bu}.

  The {\it Ramsey critical graph} $H$ for $H_{1}$ versus $H_{2}$ is a red-blue edge-colored $K_{N- 1}$ such that $H\not\rightarrow (H_{1}, H_{2})$, where $N= r(H_{1}, H_{2})$. For graphs $G$ and $H$, the join of $G$ and $H$, $G\vee H$, is the graph obtained from $G\cup H$ by adding all edges between $V(G)$ and $V(H)$. For positive integers $n, m$ and $t$, let $H_{i}= K_{n- 1}$ for each $i\in [m- 1]$ and $H_{m}= K_{t- 1}$. Moreover, let $\mathcal{G}_{n, m, t}= \{\bigvee_{i= 1}^{m} H_{i}:$ color $E(H_{i})$ by red for each $i\in [m- 1]$, color $E(V(H_{i}), V(H_{j}))$ by blue for each $i, j\in [m]$ such that $i\neq j$, and color $E(H_{m})$ arbitrarily by red and blue$\}$. One can directly verify that if $n$ is sufficiently large, then $G$ is a Ramsey critical graph for $T_{n}$ versus $tK_{m}$ for each $G\in \mathcal{G}_{n, m, t}$ since $r(T_{n}, tK_{m})= (n- 1)(m- 1)+ t$ by Theorem \ref{theo1.3}. In 2011, Hook and Isaak \cite{HI} proved that if $G$ is a Ramsey critical graph for $T_{n}$ versus $K_{m}$, then $G\in \mathcal{G}_{n, m, 1}$. Furthermore, they proved that $r_{*}(T_{n}, K_{m})= (n- 1)(m- 2)+ 1$ for $m\geq 2$. In this paper, we prove the same results for all $t$.

  \begin{theo}\label{theo1.6}
    Let $n, m$ and $t$ be positive integers, and let $G$ be a Ramsey critical graph for $T_{n}$ versus $tK_{m}$. If $n$ is sufficiently large, then $G\in \mathcal{G}_{n, m, t}$.
  \end{theo}

  \noindent Furthermore, we determine the star-critical Ramsey number for $T_{n}$ versus $tK_{m}$.

  \begin{theo}\label{theo1.7}
    Let $n, m$ and $t$ be positive integers. If $n$ is sufficiently large and $m\geq 2$, then $r_{*}(T_{n}, tK_{m})= (n- 1)(m- 2)+ t$.
  \end{theo}

  \begin{remark}
    Let $t$ be a positive integer. If $n\geq t$, then $r_{*}(T_{n}, tK_{1})= 0$.
  \end{remark}

  \textbf{Notations and definitions}:  Let $G$ be a graph. For a vertex $u\in V(G)$, let $N_{G}(u)$ be the neighbors of $u$ in $G$ and let $d_{G}(u)= |N_{G}(u)|$ be the degree of $u$ in $G$. Let $\Delta(G)$ be the maximum degree of $G$. For vertex sets $U,W\subset V(G)$ such that $U\cap W= \emptyset$, let $E_{G}(U, W)$ be the subset of $E(G)$ such that one end vertex belongs to $U$ and another belongs to $W$. Moreover, let $G- U$ be the graph obtained from $G$ by removing $U$ and all edges incident to $U$, and let $G[U]$ be the graph induced by $G$ on $U$. Let $K_{n, m}$ be the complete bipartite graph with partite sizes $n$ and $m$, respectively. For a real number $x$, let $\lfloor x\rfloor$ be the maximum integer not exceeding $x$. A vertex set $U\subset V(G)$ is an {\it independent set} if $E(G[U])= \emptyset$.

\section{For a large star $K_{1, n- 1}$}
  Firstly, we need the following important structure result, which is called the Hajnal-Szemer\'edi's Theorem.

  \begin{theo}[\rm Hajnal and Szemer\'edi \cite{HS}, Kierstead and Kostachka \cite{KK}]\label{theo2.1}
    Let $n, \ell, a$ and $b$ be positive integers such that $n= a\ell+ b$, where $0\leq b< \ell$. Let $G$ be a graph of $n$ vertices with $\Delta(G)< \ell$. Then there is a partition $A_{1}, \dots, A_{\ell}$ of $V(G)$ such that $A_{i}$ is an independent set for each $i\in [\ell]$. Furthermore, $|A_{i}|= a+ 1$ for each $i\in [b]$ and $|A_{i}|= a$ for each $i\in [\ell]\backslash [b]$.
  \end{theo}

  \noindent We also need a well-known Theorem, which is called the Hall's Theorem.

  \begin{theo}[{\rm Hall \cite{H}}]\label{theo2.2}
    Let $a$ and $b$ be positive integers such that $a\leq b$. Let $G$ be a red-blue edge-colored $K_{a, b}$ with partite sets $X$ and $Y$ such that $|X|= a$ and $|Y|= b$. Then $G$ contains a red copy of matching saturating $X$ or a blue copy of $K_{c+ 1, b- c}$ with $c+ 1$ vertices in $X$, where $0\leq c\leq a- 1$.
  \end{theo}

  \noindent Now, we are ready to prove that Theorem \ref{theo1.6} holds for a large star.

  \begin{theo}\label{lemma2.3}
    Let $n, m$ and $t$ be positive integers, and let $G$ be a Ramsey critical graph for $K_{1, n- 1}$ versus $tK_{m}$. If $n$ is sufficiently large, then $G\in \mathcal{G}_{n, m, t}$.
  \end{theo}
  \begin{proof}
    The assertion is trivial for $m= 1$ and the assertion holds for $t= 1$ from \cite{HI}. Thus, we may assume that $m\geq 2$ and $t\geq 2$. Let $G$ be a Ramsey critical graph for $K_{1,n- 1}$ versus $tK_{m}$ and let $R$ be the graph induced by the red edges of $G$. Then $R$ contains no copy of $K_{1, n- 1}$ since $G\not\rightarrow (K_{1, n- 1}, tK_{m})$, and thus $\Delta(R)< n- 1$. Note that $v(R)= v(G)= (n- 1)(m- 1)+ t- 1$. By Theorem \ref{theo2.1}, there is a partition $A_{1}, \dots, A_{n- 1}$ of $V(R)$ such that $A_{i}$ is an independent set in $R$ for each $i\in [n- 1]$. Furthermore, $|A_{i}|= m- 1$ for each $i\in [n- t]$ and $|A_{i}|= m$ for each $i\in [n- 1]\backslash [n- t]$. Consequently, $G[A_{i}]$ is a blue copy of $K_{m- 1}$ for each $i\in [n- t]$, and $G[A_{i}]$ is a blue copy of $K_{m}$ for each $i\in [n- 1]\backslash [n- t]$.

    \begin{claim}\label{claim2.4}
      There is a red copy of $(m- 1)K_{2}$ between $A_{i}$ and $A_{j}$ for each $i, j\in [n- t]$ and $i\neq j$. Furthermore, if there is a subset $D\subset [n- 1]$ of size $t- 1$ such that $|A_{k}|= m$ for each $k\in D$ and $|A_{k}|= m- 1$ for each $k\in [n-1]\backslash D$, then there is a red copy of $(m- 1)K_{2}$ between $A_{i'}$ and $A_{j'}$ for each $i', j'\in [n- 1]\backslash D$ and $i'\neq j'$.
    \end{claim}
    \begin{proof}
      Note that $E_{G}(A_{i}, A_{j})$ is a red-blue edge-colored $K_{m- 1, m- 1}$ and denote it by $H$. By Theorem \ref{theo2.2}, $H$ contains a red copy of matching saturating $A_{i}$ or a blue copy of $K_{c+ 1, m- 1- c}$ with $c+ 1$ vertices in $A_{i}$, where $0\leq c\leq m- 2$. If the latter holds, then $G[V(K_{c+ 1, m- 1- c})]$ is a blue copy of $K_{m}$ since $G[A_{i}]$ and $G[A_{j}]$ are blue copies of $K_{m}$. Together with $G[A_{n- t+ 1}], \dots, G[A_{n- 1}]$, it forms a blue copy of $tK_{m}$. It is a contradiction since $G\not\rightarrow (K_{1, n- 1}, tK_{m})$. Consequently, $H$ contains a red copy of matching saturating $A_{i}$, which is a red copy of $(m- 1)K_{2}$ since $|A_{i}|= |A_{j}|= m- 1$.\q
    \end{proof}

    \begin{claim}\label{claim2.5}
      For each $i\in [t- 1]$, there is a vertex $u_{n- i}\in A_{n- i}$ satisfying the following. There is an integer $s_{i}\in [n- t]$ such that $E_{G}(\{u_{i}\}, A_{s_{i}})$ is blue. Furthermore, $s_{i}$ are different.
    \end{claim}
    \begin{proof}
      We will find the vertex $u_{n- i}$ and $s_{i}$ step by step. Assume that there are vertices $u_{n- i}\in A_{n- i}$ such that $E_{G}(\{u_{n- i}\}, A_{i})$ is blue for each $i\in [k]$, where $k\in \{0\}\cup [t- 2]$. If there is a vertex $u\in \bigcup_{j= n- t+ 1}^{n- k- 1} A_{j}$ such that $E_{G}(\{u\}, A_{i'})$ for some $i'\in [n- t]\backslash [k]$. Then relabel the index of $A_{k+ 1}, \dots, A_{n- k- 1}$ such that $u\in A_{n- k- 1}$ and $i'= k+ 1$, and we are done.

      Otherwise, for each vertex $u\in \bigcup_{j= n- t+ 1}^{n- k- 1} A_{j}$, there is at least one vertex $v_{j}^{u}\in A_{j}$ such that $uv_{j}^{u}$ is red for each $j\in [n- t]\backslash [k]$. Thus, $\left|N_{R}(u)\cap \left(\bigcup_{j= k+ 1}^{n- t} A_{j}\right)\right|\geq n- t- k$ for each vertex $u\in \bigcup_{j= n- t+ 1}^{n- k- 1} A_{j}$. Let $w\in \bigcup_{j= k+ 1}^{n- t} A_{j}$ be a vertex. Note that there is a red copy of $(m- 1)K_{2}$ between $A_{j}$ and $A_{j'}$ for each $j, j'\in [n- t]$ and $j\neq j'$ by Claim \ref{claim2.4}. Thus, $\left|N_{R}(w)\cap \left(\bigcup_{j= 1}^{n- t} A_{j}\right)\right|\geq n- t- 1$. Recall that $G[\{u_{n- j}\}\cup A_{j}]$ is a blue copy of $K_{m}$ for each $j\in [k]$ by our assumption. Thus, there is a red copy of $(m- 1)K_{2}$ between $A_{j}$ and $A_{n- j'}\backslash \{u_{n- j'}\}$ for each $j\in [n- t]\backslash [k]$ and $j'\in [k]$ by Claim \ref{claim2.4}. Thus, $\left|N_{R}(w)\cap \left(\bigcup_{j= n- k}^{n- 1} A_{j}\right)\right|\geq k$. Note that $|N_{R}(w)|\leq n- 2$ since $\Delta(R)< n- 1$. Consequently, $\left|N_{R}(w)\cap \left(\bigcup_{j= n- t+ 1}^{n- k- 1} A_{j}\right)\right|\leq t- k- 1$. Then
      $$\begin{aligned}
        & m(t- k- 1)(n- t- k)\\
        = & \left|\bigcup_{j= n- t+ 1}^{n- k- 1} A_{j}\right|(n- t- k)\leq \sum_{u\in \bigcup_{j= n- t+ 1}^{n- k- 1} A_{j}} \left|N_{R}(u)\cap \left(\bigcup_{j= k+ 1}^{n- t} A_{j}\right)\right|\\
        = & \left|E_{R}\left(\left(\bigcup_{j= k+ 1}^{n- t} A_{j}\right), \left(\bigcup_{j= n- t+ 1}^{n- k- 1} A_{j}\right)\right)\right|\\
        = & \sum_{w\in \bigcup_{j= k+ 1}^{n- t} A_{j}} \left|N_{R}(w)\cap \left(\bigcup_{j=n- t+ 1}^{n- k- 1} A_{j}\right)\right|\leq \left|\bigcup_{j= k+ 1}^{n- t} A_{j}\right|(t- k- 1)\\
        = & (m- 1)(n- t- k)(t- k- 1).
      \end{aligned}$$
      It is a contradiction, since $m\geq 2, t\geq 2$ and $k\in \{0\}\cup [t- 2]$.\q
    \end{proof}

    Let $A_{n- i}= A_{n- i}\backslash \{u_{n- i}\}$ for each $i\in [t- 1]$ in the following for convenience. Furthermore, we may assume that $E_{G}(\{u_{n- i}\}, A_{i}\cup A_{n- i})$ is blue for each $i\in [t- 1]$ by Claim \ref{claim2.5}.

    \begin{claim}\label{claim2.6}
      There is a red copy of $(m- 1)K_{2}$ between $A_{k}$ and $A_{k'}$ for each $k, k'\in [n- 1]$ and $k\neq k'$. Moreover, $E_{G}\left(U, \bigcup_{j= 1}^{n- 1} A_{j}\right)$ is blue, where $U= \{u_{n- i}, i\in [t- 1]\}$.
    \end{claim}
    \begin{proof}
      Note that $G[\{u_{n- i}\}\cup A_{n- i}]$ is a blue copy of $K_{m}$ for each $i\in [t- 1]$. Thus, there is a red copy of $(m- 1)K_{2}$ between $A_{k}$ and $A_{k'}$ for each $k, k'\in [n- t]$ and $k\neq k'$ by Claim \ref{claim2.4}. Note that $G[\{u_{n- i}\}\cup A_{i}]$ is a blue copy of $K_{m}$ for each $i\in [t- 1]$. Thus, there is a red copy of $(m- 1)K_{2}$ between $A_{k}$ and $A_{k'}$ for each $k, k'\in [n- 1]\backslash [t- 1]$ and $k\neq k'$ by Claim \ref{claim2.4}. Thus, $\left|N_{R}(w)\cap \bigcup_{j= 1}^{n- 1} A_{j}\right|\geq n- 2$ for each vertex $w\in \bigcup_{j= t}^{n- t} A_{j}$. Consequently, $E_{G}\left(U, \bigcup_{j= t}^{n- t} A_{j}\right)$ is blue since $\Delta(R)< n- 1$. Note that $n\geq 3t- 3$ since $n$ is sufficiently large and $G[\{u_{n- i}\}\cup A_{t- 1+ i}]$ is a blue copy of $K_{m}$ for each $i\in [t- 1]$. Thus, there is a red copy of $(m- 1)K_{2}$ between $A_{k}$ and $A_{k'}$ for each $k, k'\in ([n- 1]\backslash [2t- 2])\cup [t- 1]$ and $k\neq k'$ by Claim \ref{claim2.4}. Consequently, there is a red copy of $(m- 1)K_{2}$ between $A_{k}$ and $A_{k'}$ for each $k, k'\in [n- 1]$ and $k\neq k'$. Note that $\left|N_{R}(w)\cap \bigcup_{j= 1}^{n- 1} A_{j}\right|\geq n- 2$ for each $w\in \bigcup_{j= 1}^{n- 1} A_{j}$. Consequently, $E_{G}\left(U, \bigcup_{j= 1}^{n- 1} A_{j}\right)$ is blue since $\Delta(R)< n- 1$.\q
    \end{proof}

    For each $i\in [n- 1]$, $|N_{R}(u)\cap A_{j}|= 1$ for each vertex $u\in A_{i}$ and each $j\in [n- 1]\backslash \{i\}$ by Claim \ref{claim2.6} since $\Delta(R)< n- 1$. For each $s\in [m- 1]$, let $v_{1}^{s}\in A_{1}$ and let $\{v_{j}^{s}\}= N_{R}(v_{1}^{s})\cap A_{j}$ for each $j\in [n- 1]\backslash [1]$.

    \begin{claim}\label{claim2.7}
      $G[\{v_{i}^{s}, i\in [n- 1]\}]$ is a red copy of $K_{n- 1}$ for each $s\in [m- 1]$.
    \end{claim}
    \begin{proof}
      Otherwise, there are $i', j'\in [n- 1]\backslash [1]$ such that $i'\neq j'$ and $v^{s}_{i'}v^{s}_{j'}$ is blue. Recall that $N_{R}(v^{s}_{i'})\cap A_{1}= N_{R}(v^{s}_{j'})\cap A_{1}= \{v_{1}^{s}\}$. Consequently, $G[\{v^{s}_{i'}, v^{s}_{j'}\}\cup (A_{1}\backslash\{v^{s}_{1}\})]$ is a blue copy of $K_{m}$ since $G[A_{1}]$ is a blue copy of $K_{m}$. Moreover, $G$ contains a blue copy of $tK_{m}$ since $E_{G}\left(U, \bigcup_{j= 1}^{n- 1} A_{j}\right)$ is blue by Claim \ref{claim2.6} and $n$ is sufficiently large. It is a contradiction since $G\not\rightarrow (K_{1, n- 1}, tK_{m})$.\q
    \end{proof}

    Consequently, $G\in\mathcal{G}_{n, m, t}$ by Claim \ref{claim2.6} and Claim \ref{claim2.7} since $\Delta(R)< n- 1$.\q
  \end{proof}

\section{For a large tree $T_{n}$}
  Before we prove our main theorem, we need a structure lemma as follows.
  \begin{lemma}\label{lemma3.1}
    Let $G$ be a Ramsey critical graph for $T_{n}$ versus $tK_{m}$ and let $H\in \mathcal{G}_{n, m- 1, t}$ be a graph. If $G$ contains a copy of $H$ and $n$ is sufficiently large, then $G\in \mathcal{G}_{n, m, t}$.
  \end{lemma}
  \begin{proof}
    Let $R$ and $B$ be the graphs induced by the red and blue edges of $G$, respectively. Note that $H\in \mathcal{G}_{n, m- 1, t}$, and thus $R[V(H)]= (m- 2)K_{n- 1}\cup H'= \bigcup_{i= 1}^{m- 2}H_{i}\cup H'$, where $H_{i}$ is a red copy of $K_{n- 1}$ for each $i\in [m- 2]$ and $v(H')= t- 1$. Let $V= V(G)\backslash \bigcup_{i= 1}^{m- 2} V(H_{i})$, and thus $|V|= n+ t -2$. Note that $E_{G}(V(H_{i}), V(G)\backslash V(H_{i}))$ is blue for each $i\in [m- 2]$ since $H_{i}$ is a red copy of $K_{n- 1}$ for each $i\in [m- 2]$ and $G\not\rightarrow (T_{n}, tK_{m})$. If $G[V]$ contains a red copy of $K_{n- 1}$ (denote it by $H_{m- 1}$), then $E_{G}(V(H_{m- 1}), V(G)\backslash V(H_{m- 1}))$ is blue since $G\not\rightarrow (T_{n}, tK_{m})$. Thus, $G\in \mathcal{G}_{n, m, t}$ and we are done.

    In the following, we will prove that $G[V]$ contains a red copy of $K_{n- 1}$ to finish the proof. Let $K= sK_{2}= \{u_{j}v_{j}, j\in [s]\}$ be a maximum blue matching of $G[V]$. Recall that $E_{G}\left(V, \bigcup_{i= 1}^{m- 2} V(H_{i})\right)$ is blue. If $s\geq t$, then $\{u_{j}, v_{j}\}$ and one vertex of $H_{i}$ for each $i\in [m- 2]$ form a blue copy of $K_{m}$, and $G$ contains a blue copy of $tK_{m}$ since $n$ is sufficiently large. It is a contradiction since $G\not\rightarrow (T_{n}, tK_{m})$. Consequently, $s\leq t- 1$. Let $U= V\backslash V(K)$, and thus $|U|= n+ t- 2- 2s\geq 2$ since $n$ is sufficiently large. Note that $G[U]$ is a red copy of $K_{n+ t- 2- 2s}$ since $sK_{2}$ is a maximum blue matching of $G[V]$. For convenience in the following, we may assume that $|N_{B}(u_{j})\cap U|\geq |N_{B}(v_{j})\cap U|$ for each $j\in [s]$.

    \begin{claim}\label{claim3.2}
      $|N_{B}(v_{j})\cap U|= 0$ or $|(N_{B}(u_{j})\cup N_{B}(v_{j}))\cap U|\leq 1$ for each $j\in [m- 1]$.
    \end{claim}
    \begin{proof}
      Recall that $|U|\geq 2$. If the assertion is false, then there are vertices $w, w'\in U$ such that $w\in N_{B}(u_{j})\cap U, w'\in N_{B}(v_{j})\cap U$ and $w\neq w'$. Note that $G[V]$ contains a blue copy of $(s+ 1)K_{2}$ (replace $u_{j}v_{j}$ with $wu_{j}, w'v_{j}$). It is a contradiction since $sK_{2}$ is a maximum blue matching of $G[V]$.\q
    \end{proof}

    \noindent By Claim \ref{claim3.2}, we can decompose $E(K)$ into three types as follows.

    (I) $N_{B}(u_{j})\cap U\neq \emptyset$ and $N_{B}(v_{j})\cap U= \emptyset$.

    (II) $N_{B}(u_{j})\cap U= N_{B}(v_{j})\cap U$ and $|N_{B}(u_{j})\cap U|= |N_{B}(v_{j})\cap U|= 1$.

    (III) $N_{B}(u_{j})\cap U= \emptyset$ and $N_{B}(v_{j})\cap U= \emptyset$.

    \noindent We may assume that there are $a, b$ and $c$ edges of type (I), (II) and (III) in $E(K)$, respectively. Then $a\geq 0, b\geq 0, c\geq 0$ and $a+ b+ c= s\leq t- 1$. Furthermore, we may assume that $u_{j}v_{j}$ is of type (I) for each $j\in [a]$, $u_{j}v_{j}$ is of type (II) for each $j\in [a+ b]\backslash [a]$ and $u_{j}v_{j}$ is of type (III) for each $j\in [s]\backslash [a+ b]$. Let $\{w_{j}\}= N_{B}(u_{j})\cap U= N_{B}(v_{j})\cap U$ for each $j\in [a+ b]\backslash [a]$ and let $W= \{w_{j}, j\in [a+ b]\backslash [a]\}$. Note that $x= |W|\leq b$ since $w_{j}$ may be the same. Let $A= \{v_{j}, j\in [a]\}$ and $B= \{u_{j}, j\in [s]\backslash [a]\}\cup \{v_{j}, j\in [s]\backslash [a]\}$. Note that $E_{G}(U\backslash W, A\cup B\})$ is red by the definition of the edges of type (I), (II) and (III). Furthermore, $E_{G}(U\backslash W, W)$ is red since $G[U]$ is a red copy of $K_{n+ t- 2- 2s}$. Thus, $G[V]$ contains a red copy of $K_{n+ t- 2- 2s- x}\vee (a+ 2b+ 2c+ x)K_{1}$. If $n+ t- 2- 2s- x+ a+ 2b+ 2c+ x\geq n$, then $G[V]$ contains a red copy of $T_{n}$ since $T_{n}$ is a bipartite graph and $n$ is sufficiently large. It is a contradiction since $G\not\rightarrow (T_{n}, tK_{m})$. Thus,
    $$n- 1\geq n+ t- 2- 2s- x+ a+ 2b+ 2c+ x= n+ t- 2- s+ b+ c.$$
    Consequently, $b= c= 0$ and $s= t- 1$ since $b\geq 0, c\geq 0$ and $s\leq t- 1$. Moreover, $a= s= t- 1$ since $a+ b+ c= s$.

    \begin{claim}\label{claim3.3}
      $N_{B}(u_{j})\cap U= U$ for each $j\in [s]$. Furthermore, $G[A]$ is a red copy of $K_{s}$.
    \end{claim}
    \begin{proof}
      Recall that $G[U\cup A]$ contains a red copy of $K_{n+ t- 2- 2s}\vee aK_{1}= K_{n- t}\vee (t- 1)K_{1}$ by the definition of edges of type (I). If there is $j'\in [s]$ such that $|N_{R}(u_{j'})\cap U|\geq 1$, then $G[U\cup A\cup \{u_{j'}\}]$ contains a red copy of $T_{n}$ that contains $u_{j'}$ as a leaf, since $T_{n}$ is a bipartite graph and $n$ is sufficiently large. It is a contradiction since $G\not\rightarrow (T_{n}, tK_{m})$. Consequently, $N_{R}(u_{j})\cap U= \emptyset$ and $N_{B}(u_{j})\cap U= U$ for each $j\in [s]$.

      Recall that $|U|\geq 2$. Let $w, w'\in U$ be vertices such that $w\neq w'$. Note that $wu_{j}$ and $w'u_{j}$ are blue for each $j\in [s]$ by the first part. If there are $i', j'\in [s]$ and $i'\neq j'$ such that $v_{i'}v_{j'}$ is blue, then $G[V]$ contains a blue copy of $(s+ 1)K_{2}$ (replace $u_{i'}v_{i'}, u_{j'}v_{j'}$ with $wu_{i'}, w'u_{j'}, v_{i'}v_{j'}$). It is a contradiction since $sK_{2}$ is a maximum blue matching of $G[V]$.\q
    \end{proof}

    Note that $G[U\cup A]$ is a red copy of $K_{n+ t- 2- 2s+ a}$ by the definition of the edges of type (I) and Claim \ref{claim3.3}. Thus, $G[V]$ contains a red copy of $K_{n- 1}$ since $a+ b+ c= s= t- 1$ and $b= c= 0$, and we are done.\q
  \end{proof}

  A {\it suspended path} in a graph is a path all of the internal vertices have degree two. An {\it end-edge} in a graph is an edge one of whose end vertices has degree one. A {\it talon} in a graph is a star consisting of end edges. The following lemma tells us that a large tree contains one of the above structures.

  \begin{lemma}{\rm{(Burr and Faudree \cite{BF})}}\label{lemma3.4}
    Any tree on $n$ vertices contains a suspended path on $\alpha$ vertices, or $\beta$ independent end-edges, or a talon with $\lfloor \frac{n}{4\alpha\beta}\rfloor$ edges.
  \end{lemma}

  \noindent The following structure lemma also plays an important role in our proof.

  \begin{lemma}{\rm{(Erd\H{o}s, Faudree, Rousseau and Shelp \cite{EFRH})}}\label{lemma3.5}
    Let $a, b, c$ and $d$ be positive integers such that $a\geq b(c- 1)+ d$. Consider a $K_{a+ b}$ on the vertex set $\{x_{1}, \dots, x_{a}, y_{1}, \dots, y_{b}\}$ whose edges are red-blue edge-colored. Suppose that $x_{1}x_{2}\cdots x_{a}$ is a red path that joins $x_{1}$ to $x_{a}$ and no red path that joins $x_{1}$ and $x_{a}$ has length exactly $a$. Then there is a blue copy of $K_{c}$ or $d$ of the $x_{i}$'s that are joined in blue to all $y_{j}$, where $j\in [b]$.
  \end{lemma}

  \noindent Now, we are ready to prove our main theorem.
  \\

  \noindent {\it Proof of Theorem {\rm \ref{theo1.6}}:} \;\; We use induction in $m$. It is trivial for $m= 1$. We may assume that the assertion holds for $m- 1$ and $m\geq 2$. Moreover, let $G$ be a Ramsey critical graph for $T_{n}$ versus $tK_{m}$ and let $R$ be the graph induced by the red edges of $G$. Then we need to show that $G\in \mathcal{G}_{n, m, t}$ to finish the proof. We use induction in $t$. Note that the assertion holds for $t= 1$ by \cite{HI}. We may assume that the assertion holds for $t- 1$ and $t\geq 2$. In the following, we divide the argument into three cases by Lemma \ref{lemma3.4}.

  Case 1. $T_{n}$ has a suspended path of $t(m- 1)(tm- 1)+ 3t$ vertices.

  Let $T'$ form from $T_{n}$ by shortening the suspended path by $2t$ vertices. Note that $v(G)= (n- 1)(m- 1)+ t- 1\geq (n- 1)(m- 2)+ t$, and thus $G\rightarrow (T_{n}, tK_{m- 1})$ by Theorem \ref{theo1.3}. We only need to consider the latter since $G\not\rightarrow (T_{n}, tK_{m})$. Let $Y$ be the vertex set of the blue copy of $tK_{m- 1}$. Note that $v(G- Y)= (n- 1)(m- 1)+ t- 1- t(m- 1)\geq (n- 2t- 1)(m- 1)+ t$ since $t\geq 2$ and $m\geq 2$. Thus, $G- Y\rightarrow (T', tK_{m})$ by Theorem \ref{theo1.3} since $v(T')= n- 2t$. We only need to consider the former since $G\not\rightarrow (T_{n}, tK_{m})$. Recall that the red copy of $T'$ has the suspended path of $t(m- 1)(tm- 1)+ 3t- 2t= t(m- 1)(tm- 1)+ t$ vertices and denote them by $X= \{x_{1}, x_{2}, \dots, x_{a}\}$, where $a= t(m- 1)(tm- 1)+ t$. Also, denote $Y$ by $y_{1}, y_{2}, \dots, y_{b}$, where $b= t(m- 1)$. Moreover, let $c= tm$ and $d= t$. Note that $G[X\cup Y]$ is a red-blue edge-colored $K_{a+ b}$ and $a\geq b(c- 1)+ d$. Consequently, one of the following holds in $G[X\cup Y]$ by Lemma \ref{lemma3.5}.

  (1) There is a blue copy of $K_{tm}$; (2) There are $t$ of the $x_{i}$'s that are joined in blue to all $y_{j}$, where $j\in [t(m- 1)]$; (3) The suspended path of $T'$ can be lengthened by $1$ in red, keeping the same end vertices.

  If (1) or (2) holds, then $G$ contains a blue copy of $tK_{m}$ since $tK_{m}\subset K_{tm}$ and $tK_{m}\subset tK_{1}\vee tK_{m- 1}$. It is a contradiction since $G\not\rightarrow (T_{n}, tK_{m})$. Consequently, (3) holds. Let $T''$ be the red copy of the resulting tree obtained from $T'$ by lengthening the suspended path of $T'$ by 1. Note that $v(G- T'')= (n- 1)(m- 1)+ t- 1- (n- 2t+ 1)\geq (n- 1)(m- 2)+ t$ since $t\geq 2$. Then $G- T''\rightarrow (T_{n}, tK_{m- 1})$ by Theorem \ref{theo1.3}. We only need to consider the latter since $G\not\rightarrow (T_{n}, tK_{m})$. Consequently, we get a red copy of $T''$ and a blue copy of $tK_{m- 1}$ that are vertex-disjoint. The same argument can continue using Lemma \ref{lemma3.5} until $v(T'')= n- 1$ and $G- T''\not\rightarrow (T_{n}, tK_{m- 1})$. Otherwise, $G$ contains a red copy of $T_{n}$. It is a contradiction since $G\not\rightarrow (T_{n}, tK_{m})$. Note that $v(G- T'')= (n- 1)(m- 1)+ t- 1- (n- 1)= (n- 1)(m- 2)+ t- 1$. Thus, $G- T''\in \mathcal{G}_{n, m- 1, t}$ by the induction hypothesis. Consequently, $G\in \mathcal{G}_{n, m, t}$ by Lemma \ref{lemma3.1} since $G$ is a Ramsey critical graph for $T_{n}$ and $tK_{m}$.

  Case 2. $T_{n}$ has at least $2t$ independent end-edges.

  Let $T'$ form from $T_{n}$ by removing $2t$ independent end-edges and the corresponding $2t$ end vertices. Note that $v(G)= (n- 1)(m- 1)+ t- 1$. Then $G\rightarrow (T_{n}, (t- 1)K_{m})$ by Theorem \ref{theo1.3}. We only need to consider the latter since $G\not\rightarrow (T_{n}, tK_{m})$, and let $C$ be the blue copy of $(t- 1)K_{m}$. Note that $v(G- V(C))= (n- 1)(m- 1)+ t- 1- (t- 1)m\geq (n- 1)(m- 2)+ t$ since $n$ is sufficiently large. Thus, $G- V(C)\rightarrow (T_{n}, tK_{m- 1})$ by Theorem \ref{theo1.3}. We only need to consider the latter since $G\not\rightarrow (T_{n}, tK_{m})$, and let $D$ be the blue copy of $tK_{m- 1}$. Note that $v(G- V(C)- V(D))= (n- 1)(m- 1)+ t- 1- (t- 1)m- t(m- 1)= (n- 2t)(m- 1)\geq (n- 2t- 1)(m- 1)+ 1$ since $m\geq 2$. Thus, $G- V(C)- V(D)\rightarrow (T', K_{m})$ by Theorem \ref{theo1.2} since $v(T')= n- 2t$. If the latter holds, then it together with $C$ forms a blue copy of $tK_{m}$. It is a contradiction since $G\not\rightarrow (T_{n}, tK_{m})$. Thus, we only need to consider the former. Let $X$ be the vertices of the red copy of $T'$ that are the end vertices of the removed independent end-edges, and let $Y$ be the vertices not in $T'$. Note that $|X|= 2t, |Y|= (n- 1)(m- 1)+ t- 1- (n- 2t)= (n- 1)(m- 2)+ 3t- 2, D\subset Y$ and $T'\subset G- V(C)- V(D)$. Let $a= 2t$ and $b= (n- 1)(m- 2)+ 3t- 2$. Note that $a\leq b$ since $t\geq 2$ and $E_{G}(X, Y)$ is a red-blue edge-colored $K_{a, b}$. Consequently, one of the following holds in $E_{G}(X, Y)$ by Theorem \ref{theo2.2}.

  (1) There is a red copy of matching of size $2t$ between $X$ and $Y$; (2) There is a blue copy of $K_{c+ 1, b- c}$ with $c+ 1$ vertices in $X$, where $0\leq c\leq a- 1$.

  If (1) holds, then $G$ contains a red copy of $T_{n}$. It is a contradiction since $G\not\rightarrow (T_{n}, tK_{m})$. Consequently, (2) holds. Note that every vertex in $G- V(C)- V(D)$ has at least one red edge connected to each component of $D$. Otherwise, we have a blue copy of $K_{m}$ in $G- V(C)$, and thus together with $C$ forms a blue copy of $tK_{m}$. It is a contradiction since $G\not\rightarrow (T_{n}, tK_{m})$. Consequently, $b- c\leq b- t$, and thus $c\geq t$. Moreover, $b- c\geq b- a+ 1= (n- 1)(m- 2)+ t- 1$ since $c\leq a- 1$. Consequently, $G$ contains a blue copy of $K_{t, (n- 1)(m- 2)+ t- 1}$. Let $Z$ and $Z'$ be partite sets of the blue copy of $K_{t, (n- 1)(m- 2)+ t- 1}$ such that $|Z|= (n- 1)(m- 2)+ t- 1$ and $|Z'|= t$. If $G[Z]\rightarrow (T_{n}, tK_{m- 1})$, then $G\rightarrow (T_{n}, tK_{m})$ since $tK_{m}\subset tK_{m- 1}\vee tK_{1}$. It is a contradiction since $G\not\rightarrow (T_{n}, tK_{m})$. Thus, $G[Z]\not\rightarrow (T_{n}, tK_{m- 1})$. Then $G[Z]\in \mathcal{G}_{n, m- 1, t}$ by the induction hypothesis since $|Z|= (n- 1)(m- 2)+ t- 1$. Consequently, $G\in \mathcal{G}_{n, m, t}$ by Lemma \ref{lemma3.1} since $G$ is a Ramsey critical graph for $T_{n}$ and $tK_{m}$.

  Case 3. $T_{n}$ has a talon with at least $c= \left\lfloor \frac{n}{4\cdot (t(m- 1)(tm- 1)+ 3t)\cdot 2t}\right\rfloor$ edges.

  Denote the center of the talon by $x$. Let $T'$ form from $T_{n}$ by removing $c$ end-edges and the corresponding $c$ end vertices of the talon. If $G\not\rightarrow (K_{1,n- 1}, tK_{m})$, then $G\in \mathcal{G}_{n, m, t}$ by Theorem \ref{lemma2.3}, and we are done. Thus, $G\rightarrow (K_{1, n- 1}, tK_{m})$. We only need to consider the former since $G\not\rightarrow (T_{n}, tK_{m})$. Moreover, let $y$ be the center of the red copy of $K_{1, n- 1}$. If $G- \{y\}\rightarrow (T_{n}, (t- 1)K_{m})$, then we only need to consider the latter since $G\not\rightarrow (T_{n}, tK_{m})$. Moreover, let $F$ be the blue copy of $(t- 1)K_{m}$ and let $G'= G- V(F)$. If $d_{R}(u)\geq n- c$ for each vertex $u\in V(G)$, then we can embed a red copy of $T'$ in $G$ by putting $x$ in $y$ greedily since $v(T')= n- c$. Moreover, $G$ contains a red copy of $T_{n}$. It is a contradiction since $G\not\rightarrow (T_{n}, tK_{m})$. Consequently, there is a vertex $z\in V(G)$ such that $d_{B}(z)\geq (m- 1)(n- 1)+ t- 1- (n- c)$, and thus $|N_{B}(z)\cap V(G')|\geq (n- 1)(m- 1)+ t- 1- (n- c)- (t- 1)m\geq (n- 1)(m- 2)+ 1$ since $n$ is sufficiently large. Then $G'[N_{B}(z)]$ contains a red copy of $T_{n}$ or a blue copy of $K_{m- 1}$ (it together with $z$ forms a blue copy of $K_{m}$. Together with $F$ forms a blue copy of $tK_{m}$) by Theorem \ref{theo1.2}. In any case, it is a contradiction since $G\not\rightarrow (T_{n}, tK_{m})$. Consequently, $G- \{y\}\not\rightarrow (T_{n}, (t- 1)K_{m})$. Note that $G- \{y\}\in \mathcal{G}_{n, m, t- 1}$ by the induction hypothesis. Consequently, $G\in \mathcal{G}_{n, m, t}$ since $G\not\rightarrow (T_{n}, tK_{m})$.

  All cases have been discussed, and the proof is complete.\q

\section{The star-critical Ramsey numbers}
  In this section, we determine the star-critical Ramsey number for $T_{n}$ versus $tK_{m}$. Let $H\in \mathcal{G}_{n, m, t}$ be a graph, and let $\bigcup_{i= 1}^{m}H_{i}$ be the graph induced by the red edges of $H$, where $H_{i}$ is a red copy of $K_{n- 1}$ for each $i\in [m- 1]$ and $v(H_{m})= t- 1$.
  \\

  \noindent {\it Proof of Theorem {\rm \ref{theo1.7}}:}\;\;
  We first show that $r_{*}(T_{n},tK_{m})\geq (n- 1)(m- 2)+ t$. Let $u\not\in V(H)$ be a vertex and let $N(u)= \bigcup_{i= 2}^{m} V(H_{i})$, and thus $d(u)= (n- 1)(m- 2)+ t- 1$. Moreover, color all edges between $\{u\}$ and $V(H)$ in blue and let $G'$ be the resulting graph. Note that $v(G')= (n- 1)(m- 1)+ t= r(T_{n}, tK_{m})$ and $G'\not\rightarrow (T_{n}, tK_{m})$. Consequently, $r_{*}(T_{n},tK_{m})\geq (n- 1)(m- 2)+ t$.

  In the following, we will show that $r_{*}(T_{n},tK_{m})\leq (n- 1)(m- 2)+ t$. Let $V= V(K_{N- 1})$ and $v\not\in V$ be a vertex, where $N= r(T_{n}, tK_{m})$. Moreover, add $(n- 1)(m- 2)+ t$ edges between $\{v\}$ and $V$, and let $G$ be the resulting graph. We will prove that $G\rightarrow (T_{n}, tK_{m})$. If $G[V]\rightarrow (T_{n}, tK_{m})$, then we are done. Thus, we may assume that $G[V]\not\rightarrow (T_{n}, tK_{m})$. Consequently, $G[V]\in \mathcal{G}_{n, m, t}$ by Theorem \ref{theo1.6} and we may assume that $G[V]= H$. Note that there are vertices $v_{i}\in V(H_{i})$ such that $vv_{i}\in E(G)$ for each $i\in [m- 1]$ since $d(v)= (n- 1)(m- 2)+ t$. If $vv_{i'}$ is red for some $i'\in [m- 1]$, then $G$ contains a red copy of $T_{n}$ since $H_{i'}$ is a red copy of $K_{n- 1}$, and we are done. Thus, we may assume that $vv_{i}$ is blue for each $i\in [m- 1]$. Let $A= \{v\}\cup \{v_{i}, i\in [m- 1]\}$. Note that $G[A]$ is a blue copy of $K_{m}$. Moreover, each vertex $w\in V(H_{m})$ and one vertex of $H_{i}$ for each $i\in [m- 1]$ form a blue copy of $K_{m}$. Furthermore, they are vertex-disjoint since $n$ is sufficiently large. Consequently, $G$ contains a blue copy of $tK_{m}$, and we are done.\q

\end{document}